\documentclass[11pt,reqno]{amsproc}
\usepackage[margin=1in]{geometry}
\usepackage{amsmath, amsthm, amssymb}
\usepackage[colorlinks=true, pdfstartview=FitV, linkcolor=blue,citecolor=blue, urlcolor=blue]{hyperref}
\usepackage[abbrev,lite,nobysame]{amsrefs}
\usepackage{times}
\usepackage[usenames,dvipsnames]{color}

\usepackage{mathtools}
\usepackage{dsfont}

\def\eps{\varepsilon}
\def\e{{\rm e}}

\def\dist{{\rm dist}}

\def\dd{{\rm d}}
\def\ddt{{\frac{\dd}{\dd t}}}

\def\R {\mathbb{R}}
\def\u {\boldsymbol{u}}

\def\N {\mathbb{N}}

\def\AA {{\mathcal A}}

\def\RR {{\mathcal R}}

\def\ZZ {{\mathbb Z}}

\def \l {\langle}
\def \r {\rangle}

\def\TT {{\mathbb T}^2}

\def\de{{\partial}}

\newtheorem{proposition}{Proposition}[section]
\newtheorem{theorem}[proposition]{Theorem}
\newtheorem{corollary}[proposition]{Corollary}
\newtheorem{lemma}[proposition]{Lemma}
\theoremstyle{definition}
\newtheorem{definition}[proposition]{Definition}
\newtheorem{remark}[proposition]{Remark}

\numberwithin{equation}{section}

\title[Smooth attractors for critical SQG]{Smooth attractors for weak solutions of the SQG equation with critical dissipation}

\author[M. Coti Zelati and P. Kalita]{Michele Coti Zelati and Piotr Kalita}

\address{Department of Mathematics, University of Maryland, College Park, MD 20742, USA}
\email{micotize@umd.edu}

\address{Faculty of Mathematics and Computer Science, Jagiellonian University, 30-348 Krak\'ow, Poland}
\email{piotr.kalita@ii.uj.edu.pl}



\begin{document}

\begin{abstract}
We consider the evolution of weak vanishing viscosity solutions to the critically dissipative surface 
quasi-geostrophic equation. Due to the possible non-uniqueness of solutions, we rephrase the problem 
as a set-valued dynamical system and prove the existence of a global attractor of optimal Sobolev regularity.
To achieve this, we derive a new Sobolev estimate involving H\"older norms,  which complement the
existing estimates based on commutator analysis.
\end{abstract}


\maketitle

\section{Introduction}
The forced, critically dissipative surface quasi-geostrophic (SQG) equation models the
temperature $\theta$ on the 2D boundary of a rapidly rotating half space, with small Rossby and Eckman numbers,
and constant potential vorticity (cf. \cites{Pedlosky82, CMT94}). As an initial-boundary value problem, it reads
\begin{equation}\label{eq:qge1}
\begin{cases}
\de_t\theta +\u \cdot \nabla \theta+(-\Delta)^{1/2}\theta=f,\\
\u = \RR^\perp \theta = \nabla^\perp (-\Delta)^{-1/2}\theta,\\
\theta(0)=\theta_0,\quad \int_{\TT}\theta_0(x)=0,
\end{cases}
\end{equation}
where  $\theta_0$ is the initial condition and $f$ is a time-independent, mean free force. Since its first appearance in the
mathematical literature in \cite{CMT94}, it has attracted tremendous amount of attention, in part due to striking
similarities with the three dimensional Euler and Navier--Stokes equations. We mention, without any aim of completeness,
the references \cites{Resnick95, CC04, FPV09, Dong10, DD08, CCW01} concerning various properties 
of the critical SQG equation, 
and the more recent works on the 
regularity of solutions  \cites{CV10, KNV07, CV12, KN09}.
In this paper, we analyze the space-periodic SQG equation \eqref{eq:qge1} from the longtime behavior viewpoint, 
and establish the following result.

\begin{theorem}\label{thm:main}
Let $f\in L^\infty(\TT)\cap H^{1/2}(\TT)$. The multivalued dynamical system $S(t)$ generated by \eqref{eq:qge1} 
on $L^2(\TT)$ possesses a unique global attractor $\AA$ with the
following properties:
\begin{enumerate}
	\item $S(t)\AA=\AA$ for every $t\geq0$, namely $\AA$ is invariant.
	\item $\AA$ is bounded in $H^{1/2}(\TT)$, and is thus compact in $L^2(\TT)$.
	\item For every bounded set $B\subset L^2(\TT)$,
	$$
	\lim_{t\to\infty}\dist_{L^2}(S(t)B,\AA)=0,
	$$
	where $\dist_{L^2}$ stands for the usual Hausdorff semi-distance between sets given by the
	$L^2(\TT)$ norm.
	\item $\AA$ is minimal in the class of $L^2(\TT)$-closed attracting sets and 
	maximal in the class of $L^2(\TT)$-bounded invariant sets.
\end{enumerate}
\end{theorem}
The dynamical system $S(t)$ is generated by the class of vanishing viscosity weak solutions to \eqref{eq:qge1}, 
defined as subsequential limits of solutions of a suitable family of regularized equations (see Section \ref{sec:prel}).
As these are not known to be uniquely determined by their initial condition, $\{S(t)\}_{t\in \R^+}$ is a set-valued family of 
solution operators.

Assuming $f\in L^p$, for some $p>2$, the existence of the global attractor for weak solutions has been recently established in \cite{CD14}. The important observation in that paper is that solutions to \eqref{eq:qge1} with $L^2$ initial data become 
instantaneously in $L^\infty$ and satisfy the energy \emph{equality}. This is sufficient to deduce the existence of a global
attractor that is strongly compact in $L^2$ and bounded in $L^\infty$. In the context of strong solutions 
(i.e., with $H^1$ initial data), the asymptotic behavior of \eqref{eq:qge1} has been analyzed in detail in 
\cites{CTV15,CCZV15}. In this case, the corresponding dynamical system is single-valued, and the main difficulties arise
in establishing proper dissipative estimates.

In this article, we establish the optimal Sobolev
regularity of the global attractor for vanishing viscosity weak solutions, hence improving on the result of \cite{CD14}. 
The proof is based on a new Sobolev estimate
(see Appendix \ref{app:sob}), derived by means
of suitable lower bounds on the fractional Laplacian \cite{CV12}, and which complement and to a certain extent improve the
existing estimates based on commutator analysis in the spirit of \cites{KP88, KPV91}. This method also requires a uniform H\"older
estimate, established in \cite{CCZV15}. It turns out that the same techniques yield  an even stronger result, which we
state below.

\begin{corollary}\label{cor:main}
Under the assumptions of Theorem \ref{thm:main}, $\AA$ is a bounded set of $H^1(\TT)$.
\end{corollary}

Hence, we provide a fairly complete answer to the existence and regularity of attractors of the critically dissipative  
SQG equation. Indeed, the $H^1$ regularity proven in Corollary \ref{cor:main} implies that the restriction of $S(t)$ to
$\AA$ is a well-defined, single-valued semigroup of solution operators. In other words, this work bridges the existing
works on weak solutions \cite{CD14} and strong solutions \cites{CTV15,CCZV15}, answering in the positive a question
posed in \cite{CD14}. We state this as a corollary as well.

\begin{corollary}\label{cor:main2}
Assume that $f\in L^\infty(\TT)\cap H^1(\TT)$. Then $\AA$ is a bounded set of $H^{3/2}(\TT)$.
\end{corollary}
In particular, under the assumptions of Corollary \ref{cor:main2}, the attractors of \cite{CD14}, \cites{CTV15,CCZV15}
and of Theorem \ref{thm:main} coincide.

\begin{remark}[On the forcing term $f$]
While the $H^{1/2}$ regularity of the attractor is optimal, due to the order of the dissipation in \eqref{eq:qge1},
we believe that the assumptions on $f$ could be relaxed. The minimal requirement would be $f\in L^2$. We
are skeptical that this could be the case, since even in \cite{CD14} it is assumed that $f\in L^p$, for some $p>2$.
The requirement of $f\in L^\infty$ is needed in the existence of a H\"older continuous absorbing set, while $f\in H^{1/2}$
is used in the Sobolev regularity estimate \eqref{eq:12}. Nonetheless, in the process to obtain  the $H^1$ regularity 
in Corollary \ref{cor:main}, the assumption that $f\in H^{1/2}$ seems to be quite sharp.
\end{remark}
This article is organized as follows. In the next Section \ref{sec:prel}, we recall a few facts about weak solutions to the
critically dissipative SQG equation, and set up the machinery of multivalued dynamical systems needed to analyze
their asymptotic behavior. Section \ref{sec:glob} is devoted to the proof of our main result, Theorem \ref{thm:main},
and Corollary \ref{cor:main}. In particular, we establish the existence of an absorbing set that is bounded in $H^{1/2}$. To make
the presentation self-contained, the paper
is complemented with two appendices. In the first Appendix \ref{sec:multiflow} we recall a few properties of multivalued semiflows
and their global attractors, while we prove a general Sobolev estimate in Appendix \ref{app:sob}.

\section{The SQG equation as a set-valued dynamical system}\label{sec:prel}
This section is devoted to the construction of the multivalued dynamical system $\{S(t)\}_{t\in\R^+}$ 
generated by a class of weak solutions
of \eqref{eq:qge1}. This is in fact a delicate step, that relies on the precise definition of vanishing viscosity solutions. We establish
translation and concatenation properties of vanishing viscosity solutions, and prove that the multivalued semiflow consists
of set-valued operators with closed graph. At the end, we recall a few known results from \cites{CD14,CCZV15} on the
dissipativity of $S(t)$.

\subsection{Preliminaries}
We define the two-dimensional torus as $\TT = (-\pi,\pi)^2$, and $C^\infty_{p}(\TT)$ as the space of restrictions to $\TT$ of smooth functions, $2\pi$-periodic in both variables, such that their mean value on $\TT$ is zero. In the sequel we will always use the shorthand notation for various function spaces defined on $\TT$, for example we will write $C^\infty$ for $C^\infty_{p}(\TT)$ and likewise $H^\sigma$ for the closure of $C^\infty$ in the $H^{\sigma}(\TT)$ norm. All spaces of functions defined on $\TT$ considered in the article are assumed to consist of functions which have zero mean value on $\TT$ and are $2\pi$-periodic with respect to both variables.

Given a distribution $\phi$ on $\TT$,  its Fourier coefficients are given by
$$
\widehat{\phi}_{k}=\frac{1}{(2\pi)^2}\int_{\TT} \phi(x)\e^{-ik\cdot x}\dd x\quad\text{\rm for} \quad k \in \ZZ^2_*,
$$
where $\mathbb{Z}^2_* = \mathbb{Z}^2\setminus \{0\}$. 
Using the Fourier expansion we can define the $L^2$-based Sobolev space $H^\sigma$ as 
$$
H^\sigma = \left\{ \phi\in \mathcal{D}'\,: \ \|\phi\|_{H^\sigma}^2=\sum_{k\in \mathbb{Z}^2_*} |k|^{2\sigma}|\widehat{\phi}_{k}|^2< \infty \right\}\quad \text{where}\quad \sigma\in \mathbb{R}.
$$
The scalar product in $L^2 = H^0$ will be denoted by $\l\cdot,\cdot\r$, and the same symbol is used for the duality between $H^{-\sigma}$ and $H^\sigma$. The fractional Laplacian is the nonlocal operator defined as
$$
(-\Delta)^\sigma \phi(x) = \sum_{k\in \mathbb{Z}^2_\star} |k|^{2\sigma} \widehat{\phi}_{k} \e^{ik\cdot x}\qquad \text{\rm or simply}\qquad 
\widehat{(-\Delta)^\sigma \phi}_{k} = |k|^{2\sigma} \widehat{\phi}_{k}.
$$
Note that the latter definition is valid also in the case when  $\phi$ (or $(-\Delta)^\sigma \phi$) is a periodic distribution 
not necessarily given by a function. Using the fractional Laplacian we can define equivalently the norm in $H^\sigma$ 
as $\|\phi\|_{H^{\sigma}} = \|(-\Delta)^{\sigma/2}\phi\|_{L^2}$. We define the Zygmund operator 
$\Lambda = (-\Delta)^{1/2}$, a linear and continuous operator from $H^{1/2}$ to its dual $H^{-1/2}$. 
In general, $\Lambda^{\sigma}$ is defined as $(-\Delta)^{\sigma/2}$, and for $\sigma\in (0,2)$ 
it has the real variable representation
 \begin{align*}
\Lambda^\sigma \phi(x)= c_\sigma \sum_{k \in \ZZ^2_*}  \int_{\TT}   \frac{\phi(x) - \phi(x+y)}{|y- 2\pi k|^{2+\sigma}} \dd y= c_\sigma \,\mathrm{P.V.}\int_{\R^2} \frac{\phi(x)-\phi(x+y)}{|y|^{2+\sigma}}\dd y, 
\end{align*} 
where $c_\sigma>0$ is a suitable normalization constant. Above and throughout the paper,  we will not distinguish between
functions on $\TT$ and their periodic extensions on $\R^2$.

The relation between $\u$ and $\theta$ in \eqref{eq:qge1} is given by the Riesz transform.
We define the $j$-th Riesz transform, where $j=1,2$, in terms of its Fourier coefficients as 
$$
\widehat{\mathcal{R}_j\phi}_k = \frac{i k_j}{|k|}\widehat{\phi}_k\quad\text{for}\quad k\in \mathbb{Z}^2_*. 
$$
We have 
$$
\mathcal{R}_j \phi = \de_{x_j} \Lambda^{-1}\phi.
$$
We define $\mathcal{R}^\perp = (-\mathcal{R}_2,\mathcal{R}_1)$, whereas we have $\mathcal{R}^\perp \phi = (-\partial_{x_2}\Lambda^{-1}\phi,\partial_{x_1}\Lambda^{-1}\phi)$. If $\phi \in H^{1/2}$, then 
$\mathcal{R}^\perp \phi \in H^{1/2}\times H^{1/2}$ is always divergence free.

\subsection{Vanishing viscosity solutions}
A weak solution to \eqref{eq:qge1} is a function
$$
\theta \in L^2_{loc}(0,\infty;H^{1/2}) \cap C_w([0,\infty);L^2)
$$
such that for any $\varphi\in C^\infty_0([0,\infty);C^\infty)$ there holds for every $T>0$
\begin{align}
&-\int_0^T\l\theta(t),\de_t\varphi(t)\r\, \dd t - \l\theta_0,\varphi(0)\r + \l\theta(T),\varphi(T)\r\nonumber\\
&\qquad\quad  - \int_0^T\l\theta (t), \u(t)\cdot\nabla\varphi(t)\r\dd t +\int_0^T \l\Lambda^{1/2}\theta(t),\Lambda^{1/2}\varphi(t)\r\, \dd t= \int_{0}^T\l f,\varphi(t)\r\, \dd t.\label{eq:very_weak}
\end{align}
We will work with a subclass of weak solutions, called vanishing viscosity solutions, defined as limit of the following auxiliary
problem. For $\eps>0$, consider the family of problems
\begin{equation}\label{eqn:problem_with_viscosity}
\de_t\theta^\eps - \eps \Delta \theta^\eps + \mathcal{R}^\perp \theta^\eps \cdot \nabla\theta^\eps + \Lambda \theta^\eps = f.
\end{equation}
We define the following class of vanishing viscosity solutions for the problem \eqref{eq:qge1}.
\begin{definition}\label{def:visco}
A weak solution (in the sense of \eqref{eq:very_weak})
$$
\theta\in L^2_{loc}(0,\infty;H^{1/2})\cap C_w([0,\infty);L^2)
$$ 
is a vanishing viscosity solution of \eqref{eq:qge1} if there exist
a sequence $\eps_n\searrow 0$ and corresponding solutions $\theta^{\eps_n}$ to \eqref{eqn:problem_with_viscosity}
with $\eps=\eps_n$ such that $\theta^{\eps_n}\to \theta$ in 
$C_w([0,T];L^2)$ for every $T>0$ and  $\theta^{\eps_n}(0)\to \theta(0)$ strongly in $L^2$.
\end{definition}  

Following the lines of \cite{CV10}*{Appendix C}, it is standard to prove that given $\theta_0\in L^2$ and $f\in L^\infty$, 
a (possibly non-unique) vanishing viscosity solution of \eqref{eq:qge1}
exists. Remarkably, it was shown in \cite{CD14} that all vanishing viscosity solutions are instantaneously
in $L^\infty$ (see \cite{CV10} for the unforced case) and satisfy the energy equation
\begin{equation}\label{eq:energy}
\frac{1}{2}\|\theta(t)\|_{L^2}^2 + \int_s^t \|\Lambda^{1/2}\theta(\tau)\|_{L^2}^2\, \dd\tau = \frac{1}{2}\|\theta(s)\|_{L^2}^2 + \int_s^t\l f,\theta(\tau)\r\, \dd\tau,
\end{equation}
for all $0\leq s\leq t$. This is connected to the absence of anomalous dissipation in the critical case, as
proven in \cite{CTV14}.
A straightforward yet very important consequence is that vanishing viscosity solutions are strongly
continuous. 

\begin{lemma}\label{lem:continuous}
Let $\theta$ be a vanishing viscosity  solution of \eqref{eq:qge1} with $\theta_0\in L^2$ and $f\in L^\infty$. 
Then $\theta\in C([0,\infty);L^2)$.
\end{lemma}
\begin{proof}
We know that $\theta\in C_w([0,\infty);L^2)$, whereas for $\{t_n\}\subset [0,\infty)$ and $t\in [0,\infty)$ such that $t_n\to t$ we have $\theta(t_n)\to \theta(t)$ weakly in $L^2$. The energy equation \eqref{eq:energy} implies that $\|\theta(t_n)\|_{L^2}\to \|\theta(t)\|_{L^2}$ as $n\to \infty$ and it follows that $\theta(t_n)\to \theta(t)$ strongly in $L^2$. The proof is complete. 
\end{proof}

\begin{remark}
In fact, only $f\in L^p$ for some $p>2$ is needed in the argument of \cite{CD14}. The assumption that $f\in L^\infty$ will
be used later, but we decided to impose this condition from the very beginning for the sake of clarity.
\end{remark}

\subsection{The SQG equation as an m-semiflow}
When studying the long-time dynamics of globally well-posed differential systems, a widely used approach
consists in the analysis of certain invariant sets 
of the corresponding semigroup of solution operators $S(t)$ acting on a Banach space $X$, or, more generally, on
a complete metric space (classical references are \cites{BV92,CV02,Hale88,Robinson01,SY02,Temam97}).
When uniqueness of solution is not known, the above theory is clearly not applicable, and several other
approaches have been developed in the past two decades \cites{Ball97, CV97, Cheskidov09, MV98}. 
We adopt here the approach introduced by Melnik and Valero in \cite{MV98}, which revolves
around the concept of multivalued semiflow (see Appendix \ref{sec:multiflow} for the main concepts).

The m-semiflow governed by the vanishing viscosity solutions of the surface quasi-geostrophic equation is
defined as the family $\{S(t)\}_{t\in\R^+}$ of set-valued solution operators $S(t):L^2\to \mathcal{P}(L^2)$ such that
\begin{equation}\label{eq:qge_multi}
S(t)\theta_0 = \left\{ \theta(t) :\theta  \mbox{ is a  solution given by Definition \ref{def:visco} with }  \theta(0)=\theta_0  \right\}.
\end{equation}
As a first step, we prove that $\{S(t)\}_{t\in\R^+}$ is indeed an m-semiflow, in the sense of Definition \ref{def:multi}.

\begin{lemma}\label{lem:m_semiflow}
Let $f\in L^\infty$. The family $\{S(t)\}_{t\in\R^+}$ given by \eqref{eq:qge_multi} is an m-semiflow.
\end{lemma}
\begin{proof}
The fact that of $S(t)\theta_0$ is a nonempty set follows from 
the existence of vanishing viscosity solutions for each $\theta_0\in L^2$. 
Moreover, the assertion $(i)$ of Definition \ref{def:multi} is clear as we can pass with $T$ to zero in \eqref{eq:very_weak}. 
We must show that $S(s+t)\theta_0\subset S(s)S(t)\theta_0$ for $s,t>0$ and $\theta_0\in L^2$.
To this end, take $\xi\in S(s+t)\theta_0$. There exists a vanishing viscosity solution $\theta$
such that $\theta(0)=\theta_0$ and $\theta(s+t)=\xi$. Define $\eta=\theta(t)$. 
Then $\eta\in S(t)\theta_0$. We must prove that $\xi\in S(s)\eta$, i.e., that there exists a vanishing viscosity weak solution $\overline{\theta}$ such that $\overline{\theta}(0) = \eta$ and $\overline{\theta}(s)=\xi$. 
For $\tau\in [0,\infty)$, define 
$$
\overline{\theta}(\tau) = \theta(\tau+t).
$$ 
Subtracting \eqref{eq:very_weak} written for final times $t$ and $T>t$, we find that $\overline{\theta}$ satisfies \eqref{eq:very_weak} with the initial condition $\eta$. It remains to show that $\overline{\theta}$ is the limit in $C_w([0,T];L^2)$, for each $T>0$,  
of solutions of the auxiliary problems \eqref{eqn:problem_with_viscosity} with the initial conditions 
converging strongly in $L^2$. As $\theta$ is the vanishing 
 viscosity solution, there exist sequences $\eps\to 0$ and $\theta^\eps\to \theta$ in $C_w([0,T];L^2)$ 
 such that \eqref{eqn:problem_with_viscosity} holds. Naturally, $\overline{\theta}$ is the limit of the time-translated
 approximated sequence of $\theta^\eps$. Indeed, setting
 $$
 \overline{\theta}^\eps(\tau)=\theta^\eps(\tau+t), \qquad \tau\in [0,\infty),
 $$ 
 we readily see that $\overline{\theta}^\eps$ satisfies \eqref{eqn:problem_with_viscosity} 
 for a.e. $\tau>0$ and $\overline{\theta}^\eps\to\overline{\theta}$ in $C_w([0,T-t];L^2)$ 
 we only need to prove that $\overline{\theta}^\eps(0)=\theta^\eps(t)\to \theta(t) = \overline{\theta}(0)$ 
 strongly in $L^2$. Multiplying \eqref{eqn:problem_with_viscosity} by $\theta^\eps$ and integrating 
 in $\mathbb{T}^2$ we get
\begin{equation}\label{eqn:energy_with_epsilon}
\frac{1}{2}\ddt\|\theta^\eps\|^2_{L^2} + \eps\|\theta^\eps\|_{H^1}^2 +\|\theta^\eps\|^2_{H ^{1/2}} = \l f,\theta^\eps\r,\qquad\text{\rm for a.e. } t\in (0,T).
\end{equation}
It follows by the Poincar\'e inequality that $\theta^\eps$ is bounded in $L^2(0,T;H^{1/2})$ independent of $\eps$, and
the same holds for $\sqrt\eps \theta^\eps$ in $L^2(0,T;H^1)$. Thus, from the estimate 
\begin{equation}\label{eq:nonlinear_estimate}
\left| \int_{\TT}\theta^\eps \u^\eps\cdot \nabla \varphi\, \dd x\right| \lesssim \|\theta^\eps\|_{L^4} \|\u^\eps\|_{L^4}\|\varphi\|_{H^1}\lesssim  \|\theta^\eps\|^2_{L^4} \|\varphi\|_{H^1} \lesssim \|\theta^\eps\|^2_{H^{1/2}} \|\varphi\|_{H^1},
\end{equation}
and \eqref{eqn:problem_with_viscosity}, we deduce that $\de_t\theta^\eps$ is uniformly bounded in $L^1(0,T;H^{-1})$. 
Hence by the Aubin--Lions theorem it follows that $\theta^\eps(r)\to \theta(r)$ strongly in $L^2$ for a.e. $r\in (0,T)$, 
possibly passing to a further subsequence. We take sequences $r_k\searrow t$ and $s_k\nearrow t$ such that 
this strong convergence holds.
Define 
$$
V_\eps(r) = \frac{1}{2}\|\theta^\eps(r)\|_{L^2}^2-\int_0^r\l f,\theta^\eps(s)\r \dd s
$$ 
and
$$
V(r) = \frac{1}{2}\|\theta(r)\|_{L^2}^2-\int_0^r\l f,\theta(s)\r \dd s.
$$ 
By \eqref{eqn:energy_with_epsilon} we have 
$$
V_\eps(r_k) \geq V_\eps(t)\geq V_\eps(s_k). 
$$
Passing to the limit as $\eps\to 0$ we get
$$
V(r_k) \geq \limsup_{\eps\to 0} V_\eps(t) \geq \liminf_{\eps\to 0}V_\eps(t) \geq V(s_k).
$$
Now we can take  $k$ to infinity and use the fact that $V$ is a continuous function (cf. Lemma \ref{lem:continuous})
so that we can infer that $\|\theta^\eps(t)\|_{L^2}\to \|\theta(t)\|_{L^2}$ and the proof is complete.
\end{proof}

In the next result we prove that the m-semiflow is $t_*$-closed (see Definition \ref{def:star}).

\begin{lemma}\label{lem:closed}
	Let $f\in L^\infty$. The m-semiflow given by \eqref{eq:qge_multi} is $t_*$-closed for all $t_*>0$.
\end{lemma}
\begin{proof}
	Fix any $t_*>0$. Let $\theta_0^n\to \theta_0$ strongly in $L^2$ and $\eta_n\in S(t_*)\theta_0^n$ 
	be such that $\eta_n\to \eta$ strongly in $L^2$. We must show that $\eta\in S(t_*)\theta_0$. There 
	exists a sequence $\theta^n$ of  vanishing viscosity solutions such that $\theta^n(0) = \theta_0^n$ 
	and $\theta^n(t_*) = \eta_n$. The energy equation \eqref{eq:energy} implies that the sequence 
	$\theta^n$ is bounded in $L^2(0,t_*;H^{1/2})$.
	Taking $\varphi\in C^\infty_0((0,t_*)\times \TT)$ in \eqref{eq:very_weak} we get
	$$
	-\int_0^{t_*}\l\theta^n(t),\de_t\varphi(t)\r\, \dd t =  \int_0^{t_*}\l\theta^n(t), \u(t)\cdot\nabla\varphi(t)\r\dd t -\int_0^{t_*} \l\Lambda^{1/2}\theta^n(t),\Lambda^{1/2}\varphi(t)\r\, \dd t+ \int_{0}^{t_*}\l f,\varphi(t)\r\, \dd t.
	$$
	The right-hand side of the above expression can be estimated from above by
	\begin{align*}
	&\int_0^{t_*}\|\theta^n(t)\|_{L^4} \|\u(t)\|_{L^4} \|\nabla\varphi(t)\|_{L^2}\,\dd t +\int_0^{t_*} \|\theta^n(t)\|_{H^{1/2}}\|\varphi(t)\|_{H^{1/2}}\, \dd t+ \int_{0}^{t_*}\| f\|_{L^2}\|\varphi(t)\|_{L^2}\, \dd t\\
	&\qquad \lesssim \|\theta^n\|^2_{L^2(0,t_*;H^{1/2})}\|\varphi\|_{L^\infty(0,t_*;H^1)} + \|\theta^n\|_{L^2(0,t_*;H^{1/2})}\|\varphi\|_{L^2(0,t_*;H^{1/2})}+\sqrt{t_*}\|f\|_{L^2}\|\varphi\|_{L^2(0,t_*;L^2)}.
	\end{align*}
	It follows that the distributional derivatives $\partial_t\theta^n$ belong to $L^1(0,t_*;H^{-1})$ and are uniformly bounded in that space.
	Hence, for a subsequence, not renumbered, we have $\theta^n\to \theta$ weakly in $L^2(0,t_*;H^{1/2})$ and $\partial_t\theta^n\to \partial_t \theta$ weakly in $L^1(0,t_*;H^{-1})$, for some function $\theta \in L^2(0,t_*;H^{1/2})$. 
	We can pass to the limit in all terms in \eqref{eq:very_weak} written for $\theta^n$, whereas $\theta$ satisfies \eqref{eq:very_weak} for $T=t_*$. Moreover, $\theta^n\to \theta$ in $C_w([0,t_*];L^2)$, and, as $\theta^n(t_*) \to \theta(t_*)$ weakly in $L^2$, we have $\eta=\theta(t_*)$. We must prove that $\theta$ can be obtained by a vanishing viscosity procedure on $(0,t_*)$. For each fixed $n\in \mathbb{N}$, there exist a sequence $\{\theta^{n,k}\}_{k\in\N}$ of functions satisfying \eqref{eqn:problem_with_viscosity} with $\varepsilon_k\to 0$ as $k\to\infty$ such that $\theta^{n,k}\to \theta^n$ in $C_w([0,t_*];L^2)$ and $\theta^{n,k}(0)\to \theta^n(0)$ strongly in $L^2$ as $k\to\infty$. The energy relations
	\begin{align*}
	&\frac{1}{2}\|\theta^{n}(t)\|^2_{L^2} + \int_0^t\|\theta^{n}(s)\|^2_{H^{1/2}}\, \dd s = \int_0^t\l f, \theta^{n}(s)\r \, \dd s + \frac{1}{2}\|\theta^{n}(0)\|^2_{L^2},\qquad\text{\rm for all }  t\in [0,t_*],\\
	&\frac{1}{2}\|\theta^{n,k}(t)\|^2_{L^2} + \int_0^t\|\theta^{n,k}(s)\|^2_{H^{1/2}}\, \dd s \leq  \int_0^t\l f, \theta^{n,k}(s) \r\, \dd s + \frac{1}{2}\|\theta^{n,k}(0)\|^2_{L^2},\qquad\text{\rm for all } t\in [0,t_*],
	\end{align*}
	imply that for every $n$ there exists $k_0(n)$ such that 
	\begin{equation}\label{eq:ball}
	\{\theta(t), \theta^n(t), \theta^{n,k}(t)\, :\ t\in [0,t_*], n\in \mathbb{N}, k\geq k_0(n)\} \subset B_R,
	\end{equation}
	where $R>0$ is big enough and $B_R$ denotes the closed ball in $L^2$ centered at 0 with radius $R$. 
	Notice that the topology of $C_w([0,t_*];B_R)$ is metrizable, we  corresponding metric denoted by $\dd_w$. For every $n$ we choose $k(n)\geq k_0(n)$ such that 
	$$
	\dd_w(\theta^{n,k(n)},\theta^n)\leq \frac{1}{n} \quad \text{and}\quad \|\theta^{n,k(n)}(0) - \theta^n(0)\|_{L^2}\leq \frac{1}{n}.
	$$ 
	Clearly $\theta^{n,k(n)}\to \theta$ in $C_w([0,t_*];L^2)$ and $\theta^{n,k(n)}(0)\to \theta(0)$ strongly in $L^2$. We have proved that $\theta$ is a vanishing viscosity solution on $[0,t_*]$. 	
	Since every $\theta^{n,k(n)}$ can be elongated to $[0,T]$ for arbitrary $T>t_*$ using a diagonal argument it follows that the convergence $\theta^{n,k(n)}\to \theta$ holds in $C_w([0,T];L^2)$ for all $T>0$, $\theta$ being a vanishing viscosity solution, and the proof is complete.
\end{proof}
Note that the above proof requires the strong convergence of initial conditions in the definition of the vanishing viscosity solution. Without this strong convergence we could not find the ball $B_R$ in \eqref{eq:ball}.

As a matter of fact, the m-semiflow turns out to be strict. However, notice that we have to strengthen the assumptions on the
forcing term, as we will have to exploit further parabolic regularization of solutions. The regularization follows by consecutive application of several bootstrapping results (also see Sections \ref{sec:holder} and \ref{sec:glob} below):
\begin{itemize}
	\item if the initial condition $\theta_0$ belongs to $L^2$ and $f\in L^p$ for some $p>2$, by \cite{CD14}*{Lemma 2.3} 
	$\theta(t)$ is uniformly bounded in $L^\infty$ for $t>\varepsilon$;
	\item if the initial condition $\theta_0$ belongs to $L^\infty$ and $f\in L^\infty$, by \cite{CCZV15}*{Theorem 4.2 and Remark 4.8} it follows that $\theta(t)$ is uniformly bounded in $C^\beta$ for certain $\beta=\beta(\|\theta_0\|_{L^\infty},\|f\|_{L^\infty})\in (0,1/4]$ for $t>\varepsilon$;
	\item if $f\in H^{1/2}$ then the uniform bound on $\theta(t)$ in $C^\beta$ implies the uniform bound on $\theta(t)$ in $H^{1/2}$ for $t> \varepsilon$ by Theorem \ref{prop:sob} below;
	\item if $\theta_0 \in H^{1/2}$ and $f\in L^\infty \cap H^{1/2}$ then 
	the argument of \cite{CTV15}*{Theorem 5.2}, up to a slight modification of the treatment of the 
	forcing term, implies that $\theta(t)$ becomes uniformly bounded in $H^1$ for $t>\varepsilon$.
	We discuss this modification in Section \ref{sub:smoothnes} below. 
\end{itemize}
Note that from the uniform bounds which hold for the approximate sequence of solutions to \eqref{eqn:problem_with_viscosity} one can always obtain the corresponding bounds for the solutions of \eqref{eq:very_weak}.   
Now, for the initial condition in $H^1$ the vanishing viscosity weak solution is unique (this property follows by the weak--strong uniqueness result obtained in \cite{CTV15}*{Theorem 4.4}) and the semiflow becomes single valued. Combining these results we get that if the initial condition is in $L^2$ then, instantaneously, the trajectory becomes bounded in $H^1$. In particular, if a trajectory is given on a time interval $[0,\varepsilon)$ for some small $\varepsilon>0$, then its continuation to the whole interval $[0,\infty)$ is defined uniquely. This fact is used in the following lemma.

\begin{lemma}\label{lem:strict}
Let $f\in H^{1/2}\cap L^\infty$. Then the m-semiflow $\{S(t)\}_{t\in \R^+}$ defined by \eqref{eq:qge_multi} is strict.
\end{lemma}
\begin{proof}
We must prove that $S(s)S(t)\theta_0\subset S(s+t)\theta_0$ for any arbitrary initial condition $\theta_0\in L^2$. 
Take $\xi\in S(s)S(t)\theta_0$. There exists $\eta\in S(t)\theta_0$ such that $\xi\in S(s)\eta$. We denote the corresponding 
vanishing viscosity solutions by $\theta_1$ and $\theta_2$. We have $\theta_1(0) = \theta_0$, $\theta_1(t) = \theta_2(0) = \eta$ 
and $\theta_2(s) = \xi$. Define
$$
\theta(r) = \begin{cases}
\theta_1(r)\quad&\text{if}\quad r\in [0,t),\\
\theta_2(r-t)\quad&\text{if}\quad r\in [t,\infty).
\end{cases}
$$
As $\theta(0)=\theta_0$ and $\theta(s+t)=\xi$ we only need to show that $\theta$ is a vanishing viscosity  solution. 
The fact that $\theta$ is a weak solution follows from \eqref{eq:very_weak}  written for $\theta_1$ and $\theta_2$. 
It remains to prove that $\theta$ is a vanishing viscosity limit. Notice that there exists a sequence $\theta_1^{\eps}$, 
defined on the whole positive semi axis $[0,\infty)$, of solutions to \eqref{eqn:problem_with_viscosity}  with $\eps\searrow 0$ ,
such that $\theta_1^{\eps}\to \theta_1$ in $C_w([0,t];L^2)$ and $\theta_1^{\eps}(0)\to \theta_1(0)$ strongly in $L^2$.  
We aim to show that the approximation sequence $\theta_1^\eps$ also converges to $\theta$ on $[t,T]$, for every
$T>t$. Equivalently,
defining 
$$
\overline{\theta}_1^{\eps}(\tau)= \theta_1^{\eps}(\tau+t),\qquad \tau \in[0,\infty),
$$ 
we need to prove that 
$\overline{\theta}_1^{\eps}\to \theta_2$ in $C_w([0,T-t];L^2)$. An analogous argument to that of Lemma 
\ref{lem:m_semiflow} implies that $\overline{\theta}_1^{\eps}(0) = {\theta}_1^{\eps}(t)\to \eta$ 
strongly in $L^2$. As $\eta\in S(t)\theta_0$ with $t>0$, the regularity of $f$ implies that $\eta\in H^1$ by the discussion
above. Now, for the initial condition in $H^1$ the solution given by 
Definition \ref{def:visco} with $f\in H^{1/2}\cap L^\infty$ is unique. 
By a diagonal argument, for a not renumbered subsequence,  $\overline{\theta}_1^{\eps}$ 
converges in $C_w([0,T-t];L^2)$, for all $T>t$, to a certain vanishing viscosity solution with  initial condition $\eta$. 
The weak--strong uniqueness results employed in \cite{CTV15}*{Theorem 4.4} imply that this vanishing 
viscosity solution is equal to $\theta_2$, thus completing the proof. 
\end{proof}

The requirement that $f\in H^{1/2}$ is used in the proof of Lemma \ref{lem:strict} to exploit known uniqueness 
results for solutions to \eqref{eq:qge1} with $H^1$ initial data. Without Lemma \ref{lem:strict}, we would be forced
to work with non-strict multivalued semiflows, hence obtaining an attractor that is only negatively invariant.

\subsection{H\"older continuous absorbing sets}\label{sec:holder}
The existence of a bounded absorbing set for $S(t)$ is the first step towards the proof of the
global attractor existence. A simple $L^2$ estimate on \eqref{eq:qge1} entails
\begin{equation}\label{eq:expdecayL2}
\|\theta(t)\|_{L^2}\leq \|\theta_0\|_{L^2}\e^{-c_0 t}+\frac{1}{c_0}\|f\|_{L^2} \quad \text{\rm for all}\quad t\geq 0,
\end{equation}
where $c_0>0$ is an absolute constant related to the Poincar\'e inequality. Clearly, this yields the existence of
a bounded absorbing set in $L^2$, which can be taken as a ball centered at zero and radius proportional to $\|f\|_{L^2}$. 

More importantly, there exists an absorbing set that is bounded in $L^\infty$. This was proven 
in \cite{CD14}, thanks to an explicit estimate valid for general weak solutions,
and therefore vanishing viscosity solutions as well. To be precise, any vanishing viscosity
solution to \eqref{eq:qge1} satisfies
$$
\|\theta(t)\|_{L^\infty}\leq c_p \left(\frac{\|\theta_0\|_{L^2}}{t}+\|f\|_{L^p}\right)\quad \text{\rm for all} \quad t>0,
$$
provided that $f\in L^p$, for some $p>2$. Thus, the $L^2$ norm controls the $L^\infty$ norm
for positive times, and therefore the existence of an $L^\infty$-absorbing set is a consequence
of the dissipative estimate \eqref{eq:expdecayL2}.
The scale-invariant nature of the $L^\infty$ norm 
still prevents the bootstrap to higher Sobolev estimates, but it turns out to be sufficient to prove
the energy equation and, in consequence, existence of the global attractor. To control higher order norms we need the following a priori
estimate, from \cite{CCZV15}.

\begin{proposition}[\cite{CCZV15}*{Theorem 4.2}]\label{thm:Calphaest}
Assume that $\theta_0\in L^\infty$. There exists 
$\beta=\beta(\|\theta_0\|_{L^\infty},\|f\|_{L^\infty})\in (0,1/4]$
such that
\begin{equation}\label{eq:Calpha}
\|\theta(t)\|_{C^\beta}\leq c \left[\|\theta_0\|_{L^\infty}+\|f\|_{L^\infty}\right]\quad\text{\rm for all} \quad t\geq t_\beta= \frac{3}{2(1-\beta)},
\end{equation} 
for some positive constant $c>0$.
\end{proposition}
Note that by \cite{CCZV15}*{Remark 4.8} we can make the constant $t_\beta$ in the above result arbitrarily small. The above proposition holds for vanishing viscosity solutions as well, as long as the initial datum is taken in $L^\infty$, which
is possible thanks to the existence of an absorbing set consisting of bounded functions.
Since the $\eps$-regularization of the approximating
sequence of regular solutions satisfies \eqref{eq:Calpha} (which is an estimate independent of $\eps$), standard compactness arguments imply that  \eqref{eq:Calpha} passes to the limit as $\eps\to 0$. We summarize the above considerations in
the following.

\begin{proposition}\label{thm:Calph}
Assume that $f\in L^\infty$. There exists 
$\beta=\beta(\|f\|_{L^\infty})\in (0,1/4]$ and a constant $c_1\geq 1$
such that the set
\begin{equation*}
B_{\beta}=\left\{\phi\in C^\beta: \|\phi\|_{C^\beta}\leq c_1(\|f\|_{L^\infty}+1)\right\}
\end{equation*}
is an absorbing set for $S(t)$. Moreover, 
\begin{equation}\label{eq:Calphaunif}
\sup_{t\geq 0}\sup_{\theta_0\in B_\beta}\|S(t)\theta_0\|_{C^\beta}\leq 2c_1(\|f\|_{L^\infty}+1),
\end{equation}
holds.
\end{proposition}
The uniform estimate in \eqref{eq:Calphaunif} follows from the propagation of the H\"{o}lder regularity proved in \cite{CTV15} (see also \cite{CCZV15}*{Theorem 4.1}). Now that a H\"older norm is under control, we aim to exploit Theorem \ref{prop:sob} to 
bootstrap the regularity of the absorbing sets to Sobolev spaces.

\section{The global attractor}\label{sec:glob}
As mentioned in the introduction, the existence of the global attractor of $S(t)$ is the main result of \cite{CD14}.
Recall that the global attractor is the unique compact subset of the phase space that is invariant and attracting
(see Appendix \ref{sec:multiflow}). By its minimality with respect to the attraction property, the global attractor
is contained in every absorbing set, and therefore regularity properties of absorbing sets imply analogous properties
for the global attractor. The aim of this section is to establish Sobolev regularity for the global attractor devised in \cite{CD14},
and prove that it coincide with the global attractor of strong solutions of \cites{CCZV15,CTV15} whenever the forcing
term is assumed smooth enough. Notice that Proposition \ref{thm:Calph} already establishes the H\"older continuity
of the global attractor.

\subsection{$H^{1/2}$ regularity: proof of Theorem \ref{thm:main}}
The proof of Theorem \ref{thm:main} is concluded once we are able to establish the existence of an $H^{1/2}$ estimate
for vanishing viscosity solutions of \eqref{eq:qge1}. In light of Proposition \ref{thm:Calph}, we may restrict ourselves 
to work with initial data $\theta_0\in B_\beta$, for which, in particular, we have the bound
\begin{equation}\label{eq:Lunif2}
\sup_{t\geq 0}\sup_{\theta_0\in B_\beta}\|S(t)\theta_0\|_{L^2}\lesssim \|f\|_{L^\infty} + 1.
\end{equation}
A trivial application of the energy equation \eqref{eq:energy} (in fact, only inequality is needed here) on the interval
$(t,t+1)$, for an arbitrary $t>0$, entails
$$
 \int_t^{t+1} \|\Lambda^{1/2}\theta(\tau)\|_{L^2}^2 \dd\tau \leq \frac{1}{2}\|\theta(t)\|_{L^2}^2 + \int_{t}^{t+1}\l f,\theta(\tau)\r\, \dd\tau.
$$
Consequently, using the Poincar\'e inequality and \eqref{eq:Lunif2} we infer that
\begin{equation}\label{eq:H12int}
\sup_{t\geq 0}\sup_{\theta_0\in B_\beta} \int_t^{t+1} \|S(\tau)\theta_0\|_{H^{1/2}}^2 \dd\tau \lesssim \|f\|_{L^\infty}^2 + 1.
 \end{equation}
We now make use of Theorem \ref{prop:sob}, with $\gamma=1$ and $\alpha=1/2$. Specifically, \eqref{eq:sob}
implies that
\begin{equation}\label{eq:12}
\ddt\|\theta\|^2_{H^{1/2}}+\frac14 \|\theta\|^2_{H^{1}}\lesssim \|f\|_{L^\infty}^{4/\beta}+\|f\|_{H^{1/2}}^2+1
\end{equation}
Hence
$$
\ddt\|\theta\|^2_{H^{1/2}}\lesssim \|f\|_{L^\infty}^{4/\beta}+\|f\|_{H^{1/2}}^2+1,
$$
and \eqref{eq:H12int} and the uniform Gronwall lemma implies the bound
$$
\|\theta(t)\|^2_{H^{1/2}}\lesssim \|f\|_{L^\infty}^{4/\beta}+\|f\|_{H^{1/2}}^2+1 \quad \text{\rm for all}\quad t\geq 1.
$$
and hence the existence of an $H^{1/2}$-bounded absorbing set. For later reference, we summarize everything 
in the proposition below.

\begin{proposition}\label{thm:12}
Assume that $f\in L^\infty\cap H^{1/2}$. There exist 
$\beta=\beta(\|f\|_{L^\infty})\in (0,1/4]$ and a constant $c_2\geq 1$
such that the set
\begin{equation*}
B_{1/2}=\left\{\phi\in H^{1/2}\cap C^\beta: \|\phi\|^2_{H^{1/2}\cap C^\beta}\leq c_2\left[\|f\|_{L^\infty}^{4/\beta}+\|f\|_{H^{1/2}}^2+1\right]\right\}
\end{equation*}
is an absorbing set for $S(t)$.
\end{proposition}
It is also clear from \eqref{eq:12} that the following estimate holds true
\begin{equation}\label{eq:H1int}
\sup_{t\geq 0}\sup_{\theta_0\in B_{1/2}}\left[  \|S(t)\theta_0\|_{H^{1/2}\cap C^\beta}^2  + \int_t^{t+1} \|S(\tau)\theta_0\|_{H^{1}}^2 \, \dd\tau \right] \leq c\left[\|f\|_{L^\infty}^{4/\beta}+\|f\|_{H^{1/2}}^2+1\right].
 \end{equation}
Note that this concludes the proof of Theorem \ref{thm:main}, since the invariance, minimality and maximality properties
are standard consequences of the general theory.

\subsection{Smoothness of the attractor}\label{sub:smoothnes}
The proof of Corollary \ref{cor:main} essentially follows from \eqref{eq:H1int} and the estimate
\begin{equation}\label{eq:H1}
\ddt\|\theta\|^2_{H^{1}}+\frac14 \|\theta\|^2_{H^{3/2}}\leq c\left(\|f\|_{L^\infty\cap H^{1/2}}\right).
\end{equation}
The estimate above has been established in \cite{CTV15}*{Theorem 5.2}, except there it was assumed that $f\in H^1$.
For the sake of completeness, we report here the main steps needed to derive \eqref{eq:H1}. As in \cite{CTV15}, 
we apply $\nabla$ to \eqref{eq:qge1}. After taking the (pointwise in $x$) inner product with $\nabla \theta$ we infer that
$$
(\de_t +\u\cdot\nabla + \Lambda)|\nabla \theta|^2 + D[\nabla \theta]=-2\nabla \u:\nabla\theta \cdot \nabla\theta+\nabla f\cdot \nabla \theta,
$$
where
$$
D[\nabla \theta](x)= c\int_{\R^2} \frac{\big|\nabla \theta(x)-\nabla \theta(x+y)\big|^2}{|y|^{3}}\dd y.
$$
It turns out (see \cite{CTV15}*{Theorem 5.2}) that 
\begin{equation}\label{eq:last1}
(\de_t +\u\cdot\nabla + \Lambda)|\nabla \theta|^2 + \frac14 D[\nabla \theta]\leq c\left(\|f\|_{L^\infty\cap H^{1/2}}+1\right)+\nabla f\cdot \nabla \theta,
\end{equation}
pointwise in $x$, where  $c\left(\|f\|_{L^\infty\cap H^{1/2}}+1\right)$ depends on H\"older estimates for $\theta$ which ultimately depend
only on the forcing term by \eqref{eq:H1int}. Using the fact that
$$
\frac12 \int_{\TT}D[\nabla \theta](x)\dd x=\|\theta\|^2_{H^{3/2}},
$$
we integrate \eqref{eq:last1} on $\TT$ and obtain
$$
\ddt \|\theta\|^2_{H^1} + \frac12 \|\theta\|^2_{H^{3/2}}\leq c\left(\|f\|_{L^\infty\cap H^{1/2}}+1\right)+\l\nabla f, \nabla \theta\r.
$$
Since, by interpolation,
$$
\l\nabla f, \nabla \theta\r\leq c\|f\|_{H^{1/2}}\|\theta\|^2_{H^{3/2}}\leq\frac14 \|\theta\|^2_{H^{3/2}}+c\|f\|_{H^{1/2}}^2,
$$
we finally arrive at \eqref{eq:H1}.
Notice that
a further use of the Gronwall lemma is needed here as well, combined with \eqref{eq:H1int}, 
to establish the existence of a bounded absorbing set in $H^1$.
\begin{proposition}\label{thm:H1}
Assume that $f\in L^\infty\cap H^{1/2}$. There exists
$\beta=\beta(\|f\|_{L^\infty})\in (0,1/4]$ 
such that the set
\begin{equation*}
B_{1}=\left\{\phi\in H^{1}\cap C^\beta: \|\phi\|_{H^{1}\cap C^\beta}\leq c(\|f\|_{L^\infty\cap H^{1/2}}+1)\right\}
\end{equation*}
is an absorbing set for $S(t)$.
\end{proposition}
At this point, the restriction of $S(t)$ to $B_{1}$ is a single-valued semigroup, and therefore it coincides with the
strong solution operator studied in \cites{CTV15,CCZV15}. 
 In particular, assuming $f\in H^1$, the regularity of the attractor can be further bootstrapped to $H^{3/2}$ 
 (stated in Corollary \ref{cor:main2}), and in general to higher Sobolev spaces as long as the forcing term 
 is assumed smooth enough.

\appendix

\section{Multivalued semiflows}\label{sec:multiflow}
In the following, $X$ is a Banach space and $ \R^+=[0,\infty)$. By $\mathcal{P}(X)$ we denote the family of nonempty subsets of $X$ and by $\mathcal{B}(X)$ the family of nonempty and bounded subsets of $X$. 
\begin{definition}\label{def:multi}
A family $\{S(t)\}_{t\in \R^+}$ of multivalued maps $S(t):X\to \mathcal{P}(X)$ is a multivalued semiflow if
\begin{itemize}
\item[$(i)$] for any $x\in X$ we have $S(0)x = \{x\}$;
\item[$(ii)$] for any $s,t\in\R^+$ and $x\in X$ we have $S(t+s)x \subset S(t)S(s)x$.
\end{itemize}
If in $(ii)$ we have the equality $S(t+s)x = S(t)S(s)x$ in place of inclusion then the m-semiflow is said to be \emph{strict}.
\end{definition}
A set $B_0\in \mathcal{B}(X)$ is \emph{absorbing} if for every bounded set $B\subset X$ there exists $t_B>0$ such that
$$
\bigcup_{t\geq t_B}S(t)B\subset B_0.
$$
 For fixed time, we consider the following notion of (extremely weak) continuity on $X$. 
 \begin{definition}\label{def:star}
 The m-semiflow $\{S(t)\}_{t\in \R^+}$ is said to be 
 $t_*$-closed if there exists $t_*>0$ such that the graph of $S(t_*)$ is a closed set in $X\times X$, 
 endowed with the strong topology.
 \end{definition}
As customary, the main object of study is the so-called global attractor, whose attraction property is defined in terms
of the Hausdorff semidistance in $X$, namely
$$
\mbox{dist}_X(A,B) = \sup_{x\in A}\inf_{y\in B}\|x-y\|_X.
$$ 
\begin{definition}
The set $\mathcal{A}\subset X$ is a global attractor for an m-semiflow $\{S(t)\}_{t\in \R^+}$ if 
\begin{itemize}
\item[$(i)$] $\mathcal{A}$ is a compact set in $X$. 
\item[$(ii)$] $\mathcal{A}$ is negatively semiinvariant, i.e., $\mathcal{A}\subset S(t)\mathcal{A}$ for all $t\in \R^+$.
\item[$(iii)$] $\mathcal{A}$ uniformly attracts all bounded sets in $X$, i.e., 
$$
\lim_{t\to\infty}\mbox{dist}_X(S(t)B,\mathcal{A})=0,
$$
for all $B\in \mathcal{B}(X)$.
\end{itemize}
\end{definition}
For our purposed, the following sufficient condition for the existence of the global attractor is enough.
\begin{theorem}[cf. \cite {CZ13}*{Theorem 4.6},  \cite{CZK15}*{Proposition 4.2}]\label{thm:attr_existence}
 If the multivalued semiflow $\{S(t)\}_{t\in \R^+}$ is $t_*$-closed and possesses a compact absorbing set, 
then it has a global attractor $\mathcal{A}$. It is the minimal closed uniformly attracting set. If the semiflow is strict then the attractor is moreover invariant, i.e. the equality $S(t)\mathcal{A} = \mathcal{A}$ holds for all $t\geq 0$.
\end{theorem}

The compactness requirement for the absorbing set is certainly not needed in general, as more appropriate notions
of asymptotic compactness can be shown to be equivalent to the existence of the global attractor.
The reader is referred to \cites{CZ13, CZK15, MV98, KL14} for more
details. 

The notion of $t_*$-closed dynamical system was introduced in the single-valued setting in \cite{CCP12}. In the
proof of Theorem \ref{thm:attr_existence} this property is needed exclusively to prove negative invariance 
(or invariance, when $S(t)$ is strict), while the existence of a compact, minimal, uniformly attracting set is a consequence
of the compactness (or asymptotic compactness) and dissipativity of $S(t)$.

\section{Sobolev estimates}\label{app:sob}
Let $\gamma\in (0,2)$, and consider the general SQG equation
$$
\begin{cases}
\de_t\theta +\u \cdot \nabla \theta+\Lambda^\gamma\theta=f,\\
\u = \RR^\perp \theta = \nabla^\perp \Lambda^{-1} \theta,\\
\theta(0)=\theta_0,\quad \int_{\TT}\theta_0(x)=0.
\end{cases}
$$
We aim to prove the following theorem. 

\begin{theorem}\label{prop:sob}
Let $\gamma\in (0,2)$, $\alpha\in (0,1)$ and assume that, for some 
$$
\beta\in (\max\{1-\gamma,0\}, 1)
$$  
we have the a priori control
\begin{equation}\label{eq:globalL}
[\theta(t)]_{C^\beta}\leq K_{\beta} \quad \text{\rm for all}\quad t\geq 0.
\end{equation}
Then the differential inequality
\begin{equation}\label{eq:sob}
\ddt\|\theta\|^2_{H^\alpha}+\frac14 \|\theta\|^2_{H^{\alpha+\gamma/2}}\leq c K_\beta^{\frac{4\gamma}{\gamma+\beta-1}}+c\|f\|_{H^\alpha}^2
\end{equation}
holds true for every $t\geq 0$ with a constant $c>0$.
\end{theorem}
The proof of the above theorem is split in several steps. A version of the above inequality, valid for $\gamma\in (1,2)$
and involving solely a bound on the $L^\infty$ norm of the solution was recently devised in \cite{CZ15}.

\subsection{Finite differences and nonlinear estimates}
Consider  the finite difference
\begin{align*}
\delta_h\theta(x,t)=\theta(x+h,t)-\theta(x,t),
\end{align*}
which is periodic in both $x$ and $h$, where  $x,h \in \TT$.  As in \cites{CZV14,CTV15}, it follows that
\begin{equation}\label{eq:findiff}
L (\delta_h\theta)^2+ D_\gamma[\delta_h\theta]=\delta_h f,
\end{equation}
where $L$ denotes the differential operator
$$
L=\de_t+\u\cdot \nabla_x+(\delta_h\u)\cdot \nabla_h+ \Lambda^\gamma
$$
and 
$$
D_\gamma[\psi](x)= c_\gamma \int_{\R^2} \frac{\big[\psi(x)-\psi(x+y)\big]^2}{|y|^{2+\gamma}}\dd y.
$$
For an arbitrary $\alpha\in (0,1)$,
we study the evolution of the quantity $v(x,t;h)$ defined by
$$
v(x,t;h) =\frac{\delta_h\theta(x,t)}{|h|^{1+\alpha}}.
$$
Notice that
$$
\|\theta(t)\|^2_{H^\alpha}=\int_{\R^2}\int_{\R^2}\big[v(x,t;h)\big]^2\dd h\, \dd x= 
\int_{\R^2}\int_{\R^2}\frac{\big[\theta(x+h,t)-\theta(x,t)\big]^2}{|h|^{2+2\alpha}}\dd h\, \dd x.
$$
From \eqref{eq:findiff} and a  calculation analogous to that in \cites{CZV14,CTV15}, we arrive at
\begin{align}
L v^2+\frac{ D_\gamma[\delta_h\theta] }{|h|^{2+2\alpha}}
&=-4(1+\alpha) \frac{h}{|h|^2}\cdot \delta_h\u \, v^2+
\frac{(\delta_hf)(\delta_h\theta)}{|h|^{2+2\alpha}}, \label{eq:ineq1}
\end{align}
with $\delta_h\u= \RR^\perp \delta_h\theta$. To estimate the dissipative 
term $D_\gamma[\delta_h\theta]$ from below and the nonlinear term $\delta_h\u$
from above, we report here two lemmas that were essentially proved in \cites{CV12, CZV14, CTV15,CZ15}.
The first concerns dissipation.

\begin{lemma}
There exists a positive constant
$\tilde{c}_\gamma$ such that
$$
D_\gamma[\delta_h\theta](x,t)\geq 
\tilde{c}_\gamma \frac{|\delta_h\theta(x,t)|^{2+\frac{\gamma}{1-\beta}}}{|h|^\frac{\gamma}{1-\beta}[\theta(t)]^\frac{\gamma}{1-\beta}_{C^\beta}}
$$
holds for any $x,h\in \TT$ and any $t\geq 0$.
\end{lemma}

\begin{proof}
As proven in \cites{CV12, CTV15,CZV14},
 for $r\geq 4|h|$ there holds, pointwise in $x,h$ and $t$
\begin{equation}\label{eq:dissip1}
D_\gamma[\delta_h\theta](x)\geq \frac{c_\gamma}{r^\gamma}|\delta_h\theta(x)|^2-c c_\gamma|\delta_h\theta(x)|[\theta]_{C^\beta}\frac{|h|}{r^{1+\gamma-\beta}},
\end{equation}
where $c\geq 1$ is an absolute constant and
$$
[\theta]_{C^\beta}=\sup_{x\neq y \in\TT} \frac{|\theta(x)-\theta(y)|}{|x-y|^{\beta}}.
$$
We choose $r>0$ such that
$$
\frac{c_\gamma}{r^\gamma}|\delta_h\theta(x)|^2=4c c_\gamma|\delta_h\theta(x)|[\theta]_{C^\beta}\frac{|h|}{r^{1+\gamma-\beta}},
$$
namely,
$$
r=\left[\frac{4c[\theta]_{C^\beta}|h|}{|\delta_h\theta(x)|}\right]^{\frac{1}{1-\beta}}.
$$
Notice that since $|\delta_h\theta(x)|\leq [\theta]_{C^\beta}|h|^\beta$, we immediately obtain that $r\geq 4|h|$.
The result follows by plugging $r$ back into \eqref{eq:dissip1}.
\end{proof}

Below is the pointwise estimate of the nonlinear term. Again, we only report the proof for the sake of completeness,
since a very similar estimate is found in \cites{CV12, CTV15,CZV14, CZ15}.

\begin{lemma}\label{lem:rieszbdd}
Let $r\geq 4|h|$ be arbitrarily fixed. Then
$$
|\delta_h\u(x,t)|\leq 
c\left[ r^{\gamma/2} \big(D_\gamma[\delta_h\theta](x,t)\big)^{1/2}+\frac{|h|[\theta(t)]_{C^\beta} }{r^{1-\beta}}\right],
$$
holds pointwise in $x,h\in \TT$ and $t\geq0$.
\end{lemma}

\begin{proof}
We will mostly use the representation of the Riesz transform as a singular integral, that is
\begin{align*}
\RR_j \theta(x)&= \frac{1}{2\pi} \mathrm{P.V.} \int_{\TT} \frac{y_j}{|y|^3} \theta(x+y) \dd y + \sum_{k \in \ZZ^2_*} \int_{\TT} \left( \frac{y_j + 2 \pi k_j }{|y + 2\pi k |^3}  - \frac{2 \pi k_j }{|2\pi k |^3} \right) \theta(x+y) \dd y \\
&=\frac{1}{2\pi}\,\mathrm{P.V.}\int_{\R^2}\frac{y_j}{|y|^3}\theta(x+y)\dd y.
\end{align*}
 Let us fix $r\geq 4|h|$, and write $\delta_h\u$ as
\begin{align*}
\delta_h\u(x)
=\frac{1}{2\pi} \mathrm{P.V.}\int_{\R^2} \frac{y^\perp}{|y|^3}\big[\delta_h\theta(x+y)-\delta_h\theta(x)\big]\dd y
=\delta_h\u_{in}(x)+\delta_h\u_{out}(x),
\end{align*}
where
\begin{align*}
\delta_h\u_{in}(x)=\frac{1}{2\pi}\mathrm{P.V.}\int_{\R^2} \frac{y^\perp}{|y|^3}\big[1-\chi(|y|/r)\big] \big[\delta_h\theta(x+y)-\delta_h\theta(x)\big]\dd y,
\end{align*}
and
\begin{align*}
\delta_h\u_{out}(x)&=\frac{1}{2\pi}\mathrm{P.V.}\int_{\R^2} \frac{y^\perp}{|y|^3} \chi(|y|/r)\big[\delta_h\theta(x+y)-\delta_h\theta(x)\big]\dd y\\
&=\frac{1}{2\pi}\mathrm{P.V.}\int_{\R^2} \delta_{-h}\left[\frac{y^\perp}{|y|^3}\chi(|y|/r) \right]\big[\theta(x+y)-\theta(x)\big]\dd y.
\end{align*}
Above, $\chi$ is a smooth radially non-increasing cutoff function, taken to be zero
for $|x|\leq 1$ and 1 if $|x|\geq 2$ and such that $|\chi'|\leq 2$.
For the inner piece, we obtain
\begin{align}
|\delta_h\u_{in}(x)|&\leq\frac{1}{2\pi}\int_{|y|\leq r} \frac{1}{|y|^2}|\delta_h\theta(x+y)-\delta_h\theta(x)|\dd y
\notag \\
&\leq \frac{1}{2\pi} \left[ \int_{|y|\leq r} \frac{1}{|y|^{2-\gamma}}  \right]^{1/2}\left[  \int_{\R^2}\frac{(\delta_h\theta(x+y)-\delta_h\theta(x))^2}{|y|^{2+\gamma}} \dd y\right]^{1/2}
\notag \\
&\leq cr^{\gamma/2} \big(D_\gamma[\delta_h\theta](x)\big)^{1/2}.
\label{eq:pfriesz1}
\end{align}
Regarding the outer part, the mean value theorem entails
\begin{align}\label{eq:pfriesz2}
|\delta_h\u_{out}(x)|\leq c|h|\int_{|y|\geq r/2} \frac{|\theta(x+y)-\theta(x)|}{|y|^3}\dd y
\leq c|h|[\theta]_{C^\beta}\int_{|y|\geq r/2} \frac{1}{|y|^{3-\beta}}\dd y
\leq c\frac{|h|[\theta]_{C^\beta}}{r^{1-\beta}}.
\end{align}
The conclusion follows by combining \eqref{eq:pfriesz1} and \eqref{eq:pfriesz2}. 
\end{proof}

\subsection{Proof of Theorem \ref{prop:sob}}
Without loss of generality, we may assume that $K_\beta\geq 1$.
Combining  \eqref{eq:globalL} and \eqref{eq:ineq1} with the results of the above two lemmas, we obtain the inequality
\begin{align}
L v^2+\frac12\frac{ D_\gamma[\delta_h\theta] }{|h|^{2+2\alpha}}+
\tilde{c}_\gamma \frac{|\delta_h\theta(x,t)|^{2+\frac{\gamma}{1-\beta}}}{|h|^{2+2\alpha+\frac{\gamma}{1-\beta}}K_\beta^\frac{\gamma}{1-\beta}}
&\leq c\left[ r^{\gamma/2} \big(D_\gamma[\delta_h\theta]\big)^{1/2}+\frac{|h|K_\beta }{r^{1-\beta}}\right]\frac{ v^2}{|h|}  +
\frac{(\delta_hf)(\delta_h\theta)}{|h|^{2+2\alpha}}. \label{eq:ineq2}
\end{align}
By the Cauchy-Schwartz inequality,
$$
c\left[ r^{\gamma/2} \big(D_\gamma[\delta_h\theta]\big)^{1/2}+\frac{|h|K_\beta }{r^{1-\beta}}\right]\frac{ v^2}{|h|} 
\leq \frac14\frac{ D_\gamma[\delta_h\theta]}{|h|^{2+2\alpha}} +
c \left[|h|^{2\alpha}v^4r^\gamma+ \frac{K_\beta }{r^{1-\beta}}v^2\right].
$$
We now choose $r>0$ as 
$$
r=4\left[\frac{K_\beta^2}{|h|^{2\alpha}v^2}\right]^{\frac{1}{1+\gamma-\beta}},
$$
so that, in particular by \eqref{eq:globalL} and the fact that $K_\beta\geq 1$,
$$
r=4\left[K_\beta^2\frac{|h|^{2\beta}}{|\delta_h\theta|^2}\right]^{\frac{1}{1+\gamma-\beta}}|h|^{\frac{2-2\beta}{1+\gamma-\beta}}\geq 4|h|^{\frac{2-2\beta}{1+\gamma-\beta}}\geq 4|h|,
$$
since $|h|\leq 1$ and $\gamma+\beta>1$.  In this way,
$$
|h|^{2\alpha}v^4r^\gamma+ \frac{K_\beta}{r^{1-\beta}}v^2\leq 2|h|^{2\alpha}v^4r^\gamma 
\leq  c K_\beta^{\frac{2\gamma}{1+\gamma-\beta}} |h|^{\frac{2\alpha(1-\beta)}{1+\gamma-\beta}}v^{2+\frac{2-2\beta}{1+\gamma-\beta}},
$$
and \eqref{eq:ineq2} becomes
\begin{align}
L v^2+\frac14\frac{ D_\gamma[\delta_h\theta] }{|h|^{2+2\alpha}}+
\tilde{c}_\gamma \frac{|\delta_h\theta(x,t)|^{2+\frac{\gamma}{1-\beta}}}{|h|^{2+2\alpha+\frac{\gamma}{1-\beta}}K_\beta^\frac{\gamma}{1-\beta}}
&\leq c K_\beta^{\frac{2\gamma}{1+\gamma-\beta}} |h|^{\frac{2\alpha(1-\beta)}{1+\gamma-\beta}}v^{2+\frac{2-2\beta}{1+\gamma-\beta}}+
\frac{(\delta_hf)(\delta_h\theta)}{|h|^{2+2\alpha}}. \label{eq:ineq3}
\end{align}
Using Young inequality with 
$$
p=\frac{1+\gamma-\beta}{2(1-\beta)}, \qquad q=\frac{1+\gamma-\beta}{\gamma+\beta-1},
$$
we infer that
$$
c K_\beta^{\frac{2\gamma}{1+\gamma-\beta}} |h|^{\frac{2\alpha}{1+\gamma-\beta}}v^{2+\frac{2-2\beta}{1+\gamma-\beta}}\leq 
\tilde{c}_\gamma \frac{|\delta_h\theta(x,t)|^{2+\frac{\gamma}{1-\beta}}}{|h|^{2+2\alpha+\frac{\gamma}{1-\beta}}K_\beta^\frac{\gamma}{1-\beta}} + c\frac{K_\infty^{\frac{4\gamma}{\gamma+\beta-1}}}{|h|^{2\alpha}} 
$$
Therefore, from \eqref{eq:ineq3} we deduce that
$$
L v^2+\frac14\frac{ D_\gamma[\delta_h\theta] }{|h|^{2+2\alpha}}
\leq c\frac{K_\beta^{\frac{4\gamma}{\gamma+\beta-1}}}{|h|^{2\alpha}}   +
\frac{(\delta_hf)(\delta_h\theta)}{|h|^{2+2\alpha}}. 
$$
We integrate the above inequality first in $h\in \TT$ (which is allowed, since $\alpha<1$) and then $x\in \TT$. Using that 
$$
\frac12\int_{\R^2}\int_{\R^2} \frac{ D_\gamma[\delta_h\theta] }{|h|^{2+2\alpha}}\dd h\,\dd x
=\int_{\R^2}\int_{\R^2} \frac{|\delta_h\Lambda^{\gamma/2}\theta|^2}{|h|^{2+2\alpha}}\dd h\, \dd x
=\|\theta\|^2_{H^{\alpha+\gamma/2}}
$$
and the estimate, valid for $\alpha\in (0,1)$,
\begin{align*}
\int_{\R^2}\int_{\R^2}\frac{(\delta_hf)(\delta_h\theta)}{|h|^{2+2\alpha}}\dd h\, \dd x 
&\leq \left[\int_{\R^2}\int_{\R^2}\frac{|\delta_hf|^2}{|h|^{2+2\alpha}}\dd h\, \dd x \right]^{1/2}
\left[\int_{\R^2}\int_{\R^2}\frac{|\delta_h\theta|^2}{|h|^{2+2\alpha}}\dd h\, \dd x \right]^{1/2}\\
&\leq \|f\|_{H^{\alpha}}\|\theta\|_{H^{\alpha}}
\leq \frac14 \|\theta\|_{H^{\alpha+\gamma/2}}^2+c\|f\|_{H^{\alpha}}^2,
\end{align*}
we arrive at
$$
\ddt\|\theta\|^2_{H^\alpha}+ \frac14\|\theta\|^2_{H^{\alpha+\gamma/2}}\leq cK_\beta^{\frac{4\gamma}{\gamma+\beta-1}}+  
c\|f\|_{H^{\alpha}}^2.
$$
This is precisely \eqref{eq:sob}, and the proof is concluded.

\section*{Acknowledgments}
\noindent The work of MCZ was in part supported by an AMS-Simons Travel Award. 
The work of PK was in part supported by the  Marie  Curie  International Research Staff Exchange  Scheme  Fellowship  within  the  7th  European  Community  Framework  Programme under Grant Agreement No. 295118, the International Project co-financed
by the Ministry of Science and Higher Education of Republic of Poland under grant
no. W111/7.PR/2012, and the National Science Center of Poland under Maestro Advanced Project no. DEC-2012/06/A/ST1/00262.


\bibliography{biblio}

\begin{bibdiv}
\begin{biblist}

\bib{BV92}{book}{
      author={Babin, A.~V.},
      author={Vishik, M.~I.},
       title={Attractors of evolution equations},
   publisher={North-Holland Publishing Co., Amsterdam},
        date={1992},
      volume={25},
}

\bib{Ball97}{article}{
      author={Ball, J.~M.},
       title={Continuity properties and global attractors of generalized
  semiflows and the {N}avier-{S}tokes equations},
        date={1997},
     journal={J. Nonlinear Sci.},
      volume={7},
      number={5},
       pages={475\ndash 502},
         url={http://dx.doi.org/10.1007/s003329900037},
}

\bib{CV10}{article}{
      author={Caffarelli, Luis~A.},
      author={Vasseur, Alexis},
       title={Drift diffusion equations with fractional diffusion and the
  quasi-geostrophic equation},
        date={2010},
     journal={Ann. of Math. (2)},
      volume={171},
      number={3},
       pages={1903\ndash 1930},
         url={http://dx.doi.org/10.4007/annals.2010.171.1903},
}

\bib{CCP12}{article}{
      author={Chepyzhov, Vladimir~V.},
      author={Conti, Monica},
      author={Pata, Vittorino},
       title={A minimal approach to the theory of global attractors},
        date={2012},
     journal={Discrete Contin. Dyn. Syst.},
      volume={32},
      number={6},
       pages={2079\ndash 2088},
         url={http://dx.doi.org/10.3934/dcds.2012.32.2079},
}

\bib{CV97}{article}{
      author={Chepyzhov, Vladimir~V.},
      author={Vishik, Mark~I.},
       title={Evolution equations and their trajectory attractors},
        date={1997},
     journal={J. Math. Pures Appl. (9)},
      volume={76},
      number={10},
       pages={913\ndash 964},
         url={http://dx.doi.org/10.1016/S0021-7824(97)89978-3},
}

\bib{CV02}{book}{
      author={Chepyzhov, Vladimir~V.},
      author={Vishik, Mark~I.},
       title={Attractors for equations of mathematical physics},
      series={American Mathematical Society Colloquium Publications},
   publisher={American Mathematical Society, Providence, RI},
        date={2002},
      volume={49},
}

\bib{CD14}{article}{
      author={{Cheskidov}, A.},
      author={{Dai}, M.},
       title={{The existence of a global attractor for the forced critical
  surface quasi-geostrophic equation in \$L\^{}2\$}},
        date={2014-02},
     journal={ArXiv e-prints},
      eprint={1402.4801},
}

\bib{Cheskidov09}{article}{
      author={Cheskidov, Alexey},
       title={Global attractors of evolutionary systems},
        date={2009},
     journal={J. Dynam. Differential Equations},
      volume={21},
      number={2},
       pages={249\ndash 268},
         url={http://dx.doi.org/10.1007/s10884-009-9133-x},
}

\bib{CCZV15}{article}{
      author={{Constantin}, P.},
      author={{Coti Zelati}, M.},
      author={{Vicol}, V.},
       title={{Uniformly attracting limit sets for the critically dissipative
  SQG equation}},
        date={2015},
     journal={ArXiv e-prints},
      eprint={1506.06375},
}

\bib{CCW01}{article}{
      author={Constantin, Peter},
      author={Cordoba, Diego},
      author={Wu, Jiahong},
       title={On the critical dissipative quasi-geostrophic equation},
        date={2001},
     journal={Indiana Univ. Math. J.},
      volume={50},
      number={Special Issue},
       pages={97\ndash 107},
        note={Dedicated to Professors Ciprian Foias and Roger Temam
  (Bloomington, IN, 2000)},
}

\bib{CMT94}{article}{
      author={Constantin, Peter},
      author={Majda, Andrew~J.},
      author={Tabak, Esteban},
       title={Formation of strong fronts in the {$2$}-{D} quasigeostrophic
  thermal active scalar},
        date={1994},
     journal={Nonlinearity},
      volume={7},
      number={6},
       pages={1495\ndash 1533},
         url={http://stacks.iop.org/0951-7715/7/1495},
}

\bib{CTV14}{article}{
      author={Constantin, Peter},
      author={Tarfulea, Andrei},
      author={Vicol, Vlad},
       title={Absence of anomalous dissipation of energy in forced two
  dimensional fluid equations},
        date={2014},
     journal={Arch. Ration. Mech. Anal.},
      volume={212},
      number={3},
       pages={875\ndash 903},
}

\bib{CTV15}{article}{
      author={Constantin, Peter},
      author={Tarfulea, Andrei},
      author={Vicol, Vlad},
       title={Long time dynamics of forced critical {SQG}},
        date={2015},
     journal={Comm. Math. Phys.},
      volume={335},
      number={1},
       pages={93\ndash 141},
         url={http://dx.doi.org/10.1007/s00220-014-2129-3},
}

\bib{CV12}{article}{
      author={Constantin, Peter},
      author={Vicol, Vlad},
       title={Nonlinear maximum principles for dissipative linear nonlocal
  operators and applications},
        date={2012},
     journal={Geom. Funct. Anal.},
      volume={22},
      number={5},
       pages={1289\ndash 1321},
         url={http://dx.doi.org/10.1007/s00039-012-0172-9},
}

\bib{CC04}{article}{
      author={C{\'o}rdoba, Antonio},
      author={C{\'o}rdoba, Diego},
       title={A maximum principle applied to quasi-geostrophic equations},
        date={2004},
     journal={Comm. Math. Phys.},
      volume={249},
      number={3},
       pages={511\ndash 528},
         url={http://dx.doi.org/10.1007/s00220-004-1055-1},
}

\bib{CZ15}{article}{
      author={{Coti Zelati}, M.},
       title={{Long time behavior of subcritical SQG equations in
  scale-invariant Sobolev spaces}},
        date={2015},
     journal={ArXiv e-prints},
      eprint={1512.00497},
}

\bib{CZV14}{article}{
      author={{Coti Zelati}, M.},
      author={{Vicol}, V.},
       title={{On the global regularity for the supercritical SQG equation}},
        date={2014},
     journal={ArXiv e-prints},
      eprint={1410.3186},
}

\bib{CZ13}{article}{
      author={Coti~Zelati, Michele},
       title={On the theory of global attractors and {L}yapunov functionals},
        date={2013},
     journal={Set-Valued Var. Anal.},
      volume={21},
      number={1},
       pages={127\ndash 149},
         url={http://dx.doi.org/10.1007/s11228-012-0215-2},
}

\bib{CZK15}{article}{
      author={Coti~Zelati, Michele},
      author={Kalita, Piotr},
       title={Minimality properties of set-valued processes and their pullback
  attractors},
        date={2015},
     journal={SIAM J. Math. Anal.},
      volume={47},
      number={2},
       pages={1530\ndash 1561},
         url={http://dx.doi.org/10.1137/140978995},
}

\bib{Dong10}{article}{
      author={Dong, Hongjie},
       title={Dissipative quasi-geostrophic equations in critical {S}obolev
  spaces: smoothing effect and global well-posedness},
        date={2010},
     journal={Discrete Contin. Dyn. Syst.},
      volume={26},
      number={4},
       pages={1197\ndash 1211},
}

\bib{DD08}{article}{
      author={Dong, Hongjie},
      author={Du, Dapeng},
       title={Global well-posedness and a decay estimate for the critical
  dissipative quasi-geostrophic equation in the whole space},
        date={2008},
     journal={Discrete Contin. Dyn. Syst.},
      volume={21},
      number={4},
       pages={1095\ndash 1101},
}

\bib{FPV09}{article}{
      author={Friedlander, Susan},
      author={Pavlovi{\'c}, Nata{\v{s}}a},
      author={Vicol, Vlad},
       title={Nonlinear instability for the critically dissipative
  quasi-geostrophic equation},
        date={2009},
     journal={Comm. Math. Phys.},
      volume={292},
      number={3},
       pages={797\ndash 810},
}

\bib{Hale88}{book}{
      author={Hale, Jack~K.},
       title={Asymptotic behavior of dissipative systems},
      series={Mathematical Surveys and Monographs},
   publisher={American Mathematical Society, Providence, RI},
        date={1988},
      volume={25},
}

\bib{KL14}{article}{
      author={Kalita, Piotr},
      author={{\L}ukaszewicz, Grzegorz},
       title={Global attractors for multivalued semiflows with weak continuity
  properties},
        date={2014},
     journal={Nonlinear Anal.},
      volume={101},
       pages={124\ndash 143},
}

\bib{KP88}{article}{
      author={Kato, Tosio},
      author={Ponce, Gustavo},
       title={Commutator estimates and the {E}uler and {N}avier-{S}tokes
  equations},
        date={1988},
     journal={Comm. Pure Appl. Math.},
      volume={41},
      number={7},
       pages={891\ndash 907},
         url={http://dx.doi.org/10.1002/cpa.3160410704},
}

\bib{KPV91}{article}{
      author={Kenig, Carlos~E.},
      author={Ponce, Gustavo},
      author={Vega, Luis},
       title={Well-posedness of the initial value problem for the {K}orteweg-de
  {V}ries equation},
        date={1991},
     journal={J. Amer. Math. Soc.},
      volume={4},
      number={2},
       pages={323\ndash 347},
         url={http://dx.doi.org/10.2307/2939277},
}

\bib{KN09}{article}{
      author={Kiselev, A.},
      author={Nazarov, F.},
       title={A variation on a theme of {C}affarelli and {V}asseur},
        date={2009},
     journal={Zap. Nauchn. Sem. S.-Peterburg. Otdel. Mat. Inst. Steklov.
  (POMI)},
      volume={370},
      number={Kraevye Zadachi Matematicheskoi Fiziki i Smezhnye Voprosy Teorii
  Funktsii. 40},
       pages={58\ndash 72, 220},
         url={http://dx.doi.org/10.1007/s10958-010-9842-z},
}

\bib{KNV07}{article}{
      author={Kiselev, A.},
      author={Nazarov, F.},
      author={Volberg, A.},
       title={Global well-posedness for the critical 2{D} dissipative
  quasi-geostrophic equation},
        date={2007},
     journal={Invent. Math.},
      volume={167},
      number={3},
       pages={445\ndash 453},
}

\bib{MV98}{article}{
      author={Melnik, Valery~S.},
      author={Valero, Jos{\'e}},
       title={On attractors of multivalued semi-flows and differential
  inclusions},
        date={1998},
     journal={Set-Valued Anal.},
      volume={6},
      number={1},
       pages={83\ndash 111},
         url={http://dx.doi.org/10.1023/A:1008608431399},
}

\bib{Pedlosky82}{book}{
      author={Pedlosky, Joseph},
       title={Geophysical fluid dynamics},
   publisher={Springer, Berlin},
        date={1982},
}

\bib{Resnick95}{book}{
      author={Resnick, Serge~G.},
       title={Dynamical problems in non-linear advective partial differential
  equations},
   publisher={ProQuest LLC, Ann Arbor, MI},
        date={1995},
        note={Thesis (Ph.D.)--The University of Chicago},
}

\bib{Robinson01}{book}{
      author={Robinson, James~C.},
       title={Infinite-dimensional dynamical systems},
      series={Cambridge Texts in Applied Mathematics},
   publisher={Cambridge University Press, Cambridge},
        date={2001},
         url={http://dx.doi.org/10.1007/978-94-010-0732-0},
        note={An introduction to dissipative parabolic PDEs and the theory of
  global attractors},
}

\bib{SY02}{book}{
      author={Sell, George~R.},
      author={You, Yuncheng},
       title={Dynamics of evolutionary equations},
      series={Applied Mathematical Sciences},
   publisher={Springer-Verlag, New York},
        date={2002},
      volume={143},
         url={http://dx.doi.org/10.1007/978-1-4757-5037-9},
}

\bib{Temam97}{book}{
      author={Temam, Roger},
       title={Infinite-dimensional dynamical systems in mechanics and physics},
     edition={Second},
      series={Applied Mathematical Sciences},
   publisher={Springer-Verlag, New York},
        date={1997},
      volume={68},
         url={http://dx.doi.org/10.1007/978-1-4612-0645-3},
}

\end{biblist}
\end{bibdiv}


\end{document}